\documentclass[12pt,fleqn,leqno]{amsart}
\usepackage{amsfonts,amssymb}
\usepackage{mathrsfs}
\usepackage{enumitem}

\usepackage[svgnames]{xcolor}
\usepackage[colorlinks,linkcolor=blue,citecolor=Green]{hyperref}
\usepackage{titlesec}
\usepackage{tikz-cd}
\titleformat{\section}[block]
 {\bfseries}
 {\thesection.}
 {\fontdimen2\font}
 {}

\setlist{noitemsep}
\setenumerate{labelindent=\parindent,label=\upshape{(\alph*)}}

\newtheorem{theorem}{Theorem}[section]
\newtheorem{corollary}[theorem]{Corollary}

\newtheorem{proposition}[theorem]{Proposition}

\theoremstyle{definition}

\newtheorem{example}[theorem]{Example}
\newtheorem{question}{Question}

\renewcommand{\emptyset}{\varnothing}
\numberwithin{equation}{section}

\DeclareMathOperator{\uhr}{\upharpoonright}
\DeclareMathOperator{\leqt}{\trianglelefteq} 
\newcommand{\cmpl}{{^{_\complement}}}
\newcommand{\dblr}[1]{\langle\mkern-4mu\langle{#1}
  \rangle\mkern-4mu\rangle_{\leqt_f}}  

\DeclareMathAlphabet{\mathpzc}{OT1}{pzc}{m}{it}
\DeclareMathOperator{\sel}{\mathpzc{V\mkern-5mu_{cs}}}

\begin{document}

\author{Valentin Gutev}

\address{Institute of Mathematics and Informatics, Bulgarian Academy
  of Sciences, Acad. G. Bonchev Street, Block 8, 1113 Sofia, Bulgaria}

\email{\href{mailto:gutev@math.bas.bg}{gutev@math.bas.bg}}
   


\title{Ordinal Decompositions and Extreme Selections}

\begin{abstract}
  The paper contains two natural constructions of extreme hyperspace
  selections generated by special ordinal decompositions of the
  underlying space. These constructions are very efficient not only
  in simplifying arguments but also in clarifying the ideas behind
  several known results. They are also crucial in obtaining some
  new results for such extreme selections. This is achieved by using
  special sets called clopen modulo a point. Such sets are naturally
  generated by a relation between closed sets and points of the space
  with respect to a given hyperspace selection.
\end{abstract}

\date{\today}
\maketitle

\section{Introduction}

All spaces in this paper are Hausdorff topological spaces. For a space
$X$, let $\mathscr{F}(X)$ be the set of all nonempty closed subsets of
$X$. Also, for a family $\mathscr{V}$ of subsets of $X$, let
\begin{equation}
  \label{eq:Min-Sel-v35:6}
  \langle\mathscr{V}\rangle = \left\{S\in \mathscr{F}(X) : S\subset
  \bigcup \mathscr{V}\ \ \text{and}\ \ S\cap V\neq \emptyset,\
  \hbox{whenever}\ V\in \mathscr{V}\right\}.
\end{equation}
In case $\mathscr{V}=\{V_1,\dots, V_n\}$, we will simply write
$\langle V_1,\dots, V_n\rangle$.\medskip 

Usually, we endow $\mathscr{F}(X)$ with the \emph{Vietoris topology}
$\tau_V$, and call it the \emph{Vietoris hyperspace} of $X$. Recall
that $\tau_V$ is generated by all collections of the form
$ \langle\mathscr{V}\rangle $, where $\mathscr{V}$ runs over the
finite families of open subsets of $X$. In the sequel, any subset
$\mathscr{D}\subset \mathscr{F}(X)$ will carry the relative Vietoris
topology $\tau_V$ as a subspace of the hyperspace
$(\mathscr{F}(X),\tau_V)$.  A map $f:\mathscr{D}\to X$ is a
\emph{selection} for $\mathscr{D}$ if $f(S)\in S$ for every
$S\in\mathscr{D}$. A selection $f:\mathscr{D}\to X$ is
\emph{continuous} if it is continuous with respect to the relative
Vietoris topology $\tau_V$ on $\mathscr{D}$, and we use
$\sel[\mathscr{D}]$ to denote the set of all \emph{Vietoris continuous
  selections} for $\mathscr{D}$. \medskip

It was shown by Michael in \cite[Lemmas 7.2 and 7.3]{michael:51} that
if $X$ is a connected space, $f\in \sel[\mathscr{F}(X)]$ and $p=f(X)$,
then $f(S)=p$ for every $S\in \mathscr{F}(X)$ with $p\in S$. Motivated
by this property, for an arbitrary space $X$ and a point $p\in X$, a
selection $f\in \sel[\mathscr{F}(X)]$ is called \emph{$p$-maximal}
\cite{garcia-ferreira-gutev-nogura-sanchis-tomita:99} (see also
\cite{gutev-nogura:00b}) if $f(S)=p$ for every $S\in \mathscr{F}(X)$
with $p\in S$. In this case, we also say that the point $p\in X$ is
\emph{selection-maximal}, and refer to the selection $f$ as
\emph{point-maximal}. Similarly, a selection
$f\in\sel[\mathscr{F}(X)]$ is called \emph{$p$-minimal}
\cite{garcia-ferreira-gutev-nogura-sanchis-tomita:99} if $f(S)\neq p$
for every $S\in \mathscr{F}(X)$ with $S\neq\{p\}$. Just like before,
we say that $f$ is \emph{point-minimal} and refer to the corresponding
point as \emph{selection-minimal}. In these terms, a connected space
$X$ can only have point-maximal selections for $\mathscr{F}(X)$, some
of which could be also point-minimal, see e.g. \cite[Lemma
10]{nogura-shakhmatov:97a}. Furthermore, as shown by Michael in
\cite[Proposition 7.8]{michael:51}, such a space $X$ can have at most
two continuous selections for $\mathscr{F}(X)$. This result was
refined in \cite[Theorem 1]{nogura-shakhmatov:97a} where it was shown
that a non-degenerate connected space $X$ has exactly two continuous
selections for $\mathscr{F}(X)$ if and only if it is compact and
orderable. For a space $X$ with $\sel[\mathscr{F}(X)]\neq \emptyset$,
another refinement of \cite[Proposition 7.8]{michael:51} was obtained
in \cite[Corollary
2.3]{garcia-ferreira-gutev-nogura-sanchis-tomita:99}, namely that $X$
has at most two continuous selections for $\mathscr{F}(X)$ if and only
if it has only point-maximal selections for
$\mathscr{F}(X)$. Moreover, as shown in \cite[Theorem
2.1]{garcia-ferreira-gutev-nogura-sanchis-tomita:99}, such a space $X$
can have at most two connected components.\medskip

The case when $\mathscr{F}(X)$ has ``many'' point-maximal selections
was studied in \cite{gutev-nogura:00b, gutev-nogura:03a}. A space $X$
is called \emph{selection pointwise-maximal} \cite{gutev-nogura:03a}
if each point $p\in X$ is selection-maximal.  It was shown in the
proof of \cite[Theorems 1.4]{gutev-nogura:00b} that each selection
pointwise-maximal space is zero-dimensional. Subsequently, a complete
description of these spaces was given in \cite[Theorem
4.4]{gutev-nogura:03a}. In contrast, selection-minimal points were not
studied in a systematic way. In \cite{gutev:00e,gutev-nogura:00d} they
played an interesting role in extending selections to one-point
compactifications in locally compact spaces, and in
\cite{Gutev2022a}\,---\,for hyperspace selections avoiding
points. Selection-minimal points were also studied in
\cite{MR3122363}. Such points allow us to consider yet another extreme
case, namely, we will say that a space $X$ is \emph{selection
  pointwise-minimal} if each point of $X$ is selection-minimal. This
property is also naturally related to disconnectedness. Indeed, as
shown in \cite[Proposition 4.4]{MR3122363}, each selection
pointwise-minimal space is totally disconnected.\medskip

For a space $X$, following \cite{gutev-nogura:03a}, we will say that a
nonempty closed subset $H\subset X$ is \emph{clopen modulo a point} if
$H\setminus\{p\}$ is open for some point $p\in H$. When clarity seems
to demand it, we will also say that $H$ is \emph{clopen modulo the
  point $p$}. The following family will play a crucial role in all our
considerations:
\begin{equation}
  \label{eq:Min-Sel-v37:1}
  \Delta(X)=\left\{H\in \mathscr{F}(X): H\ \text{is clopen modulo a
      point} \right\}. 
\end{equation}
Turning to examples of such sets, let us recall that a subset
$H\subset X$ is called \emph{almost clopen} \cite{Dow2018} if it is
the closure of an open subset of $X$ and has a unique boundary
point. Evidently, the family $\Delta(X)$ contains all almost clopen
subsets of $X$. This family also contains all singletons of $X$.
Similarly, $\Delta(X)$ contains all nonempty clopen subsets of $X$.
In fact, if a subset $H\subset X$ is clopen modulo points $p,q\in H$
with $p\neq q$, then it must be clopen. This follows easily from the
fact that, in this case, $H=(H\setminus\{p\})\cup (H\setminus\{q\})$
and therefore $H$ is open as well. Thus, if $H\in \Delta(X)$ is not
clopen, then it is clopen modulo a unique point $p\in H$.\medskip

Sets clopen modulo a point have been implicitly present in various
results about extreme-like hyperspace selections, see e.g.\
\cite{garcia-ferreira-gutev-nogura-sanchis-tomita:99,
  gutev:05a,MR3712970,Gutev2024a,gutev-nogura:03a, jiang:10}, but have
not been explicitly defined. This has led to arguments that are
somewhat technically demanding. We can now state also the main purpose
of this paper. Namely, in this paper we will show that such sets,
similar to clopen, are instrumental in simplifying arguments and
clarifying the ideas behind several of these results. They are also
instrumental in obtaining some new results for such extreme
selections.\medskip

In the next section, we consider special decompositions
${\mathscr{D}\subset \Delta(X)}$ of a space $X$ for which the
decomposition space $\mathscr{D}$ is homeomorphic to a compact ordinal
space, and call them \emph{ordinal decompositions}. In the same
section, in Theorems \ref{theorem-Min-Sel-v35:1} and
\ref{theorem-Min-Sel-v35:2}, we show how such decompositions generate
point-maximal and point-minimal selections for $\mathscr{F}(X)$. In
\eqref{eq:Min-Sel-v27:1} and \eqref{eq:Min-Sel-v27:2} of Section
\ref{sec:select-relat-sets}, we describe a natural construction of
members of the family $\Delta(X)$ related to a given hyperspace
selection. This construction gives a very simple proof of two known
results, see Corollaries \ref{corollary-Min-Sel-v32:1} and
\ref{corollary-Min-Sel-v18:1}. In Section \ref{sec:cut-points-first},
we consider general cut points of a space $X$ that correspond exactly
to the almost clopen members of the family $\Delta(X)$. Next, we give
a simple direct proof that a space $X$ is first countable and
zero-dimensional at each selection-maximal cut point, see Proposition
\ref{proposition-Ord-Dec-vgg:1} and Corollary
\ref{corollary-Min-Sel-v25:1}. The original proof of this fact in
\cite{gutev-nogura:03a} is not only technically demanding, but also
based on other results. Things are culminating in the last two
sections of this paper. In Section \ref{sec:select-pointw-maxim}, we
give a very natural construction of ordinal decompositions generated
by point-maximal selections. In particular, we give a very simple
proof of the topological characterisation of selection
pointwise-maximal spaces in \cite{gutev-nogura:03a}, see Theorem
\ref{theorem-point-extreme-complete}. In Section
\ref{sec:point-extr-select}, as an application of ordinal
decompositions, we show that for each selection pointwise-maximal
pseudocompact space $X$, the Vietoris hyperspace
$\mathscr{F}(\beta X)$ of the {\v C}ech-Stone compactification
$\beta X$ of $X$ has a $p$-maximal selection for each point $p\in X$
(Theorem \ref{theorem-Min-Sel-v40:2}). Also, we show that each
selection pointwise-maximal space is selection pointwise-minimal
(Theorem \ref{theorem-sa_points-v9:2}).

\section{Selections Generated by Ordinal Decompositions}

For an ordinal $\gamma$, we often identify $[0,\gamma]$ with
$\gamma+1$. In order to avoid any misunderstanding, we will not use
this convention. In this setting, the \emph{standard $\gamma$-maximal
  selection} for $\mathscr{F}([0,\gamma])$ is defined by
\begin{equation}
  \label{eq:Min-Sel-v35:2}
  \mathscr{F}([0,\gamma])\ni S\quad \longrightarrow\quad \max S\in [0,\gamma]. 
\end{equation}
Similarly, the \emph{standard $\gamma$-minimal selection} for
$\mathscr{F}([0,\gamma])$ is just expressing the property that
$[0,\gamma]$ is a well-ordered set, namely
\begin{equation}
  \label{eq:Min-Sel-v35:4}
  \mathscr{F}([0,\gamma])\ni S\quad \longrightarrow\quad \min S\in [0,\gamma]. 
\end{equation}

In this section, we will extend these constructions to a more general
situation. For this purpose, we will say that
$\mathscr{D}\subset \Delta(X)$ is an \emph{ordinal decomposition} of a
space $X$ if it is an upper semi-continuous decomposition such that
the decomposition space $\mathscr{D}$ is homeomorphic to a compact
ordinal space $[0,\gamma]$. Thus, identifying $\mathscr{D}$ with
$[0,\gamma]$, an ordinal decomposition $\mathscr{D}$ is identical to a
quotient closed map $\eta:X\to [0,\gamma]$ such that
$\mathscr{D}=\left\{\eta^{-1}(\alpha): \alpha\leq
  \gamma\right\}\subset \Delta(X)$.\medskip

Let $\eta:X\to [0,\gamma]$ be an ordinal decomposition of a space
$X$. Using the $\gamma$-max\-imal selection for
$\mathscr{F}([0,\gamma])$ in \eqref{eq:Min-Sel-v35:2}, to each family
$g_\alpha\in \sel[\mathscr{F}(\eta^{-1}(\alpha))]$,
$\alpha\leq \gamma$, of selections we will associate the selection
$f=\bigvee_{\alpha\leq \gamma} g_\alpha$ for $\mathscr{F}(X)$ defined
by
\begin{equation}
  \label{eq:Min-Sel-v35:3}
  f(S)=g_{\max \eta(S)}\left(\eta^{-1}(\max \eta(S))\right)\quad
  \text{for every $S\in 
    \mathscr{F}(X)$.} 
\end{equation}
The idea of this construction is illustrated in the diagram
below:

\begin{center}
\begin{tikzcd}[row sep=1em, column sep=2.5em, /tikz/la/.style={label={[rotate=90,anchor=west,inner
sep=1pt,label={right:\mathscr{F}(X)}]:\in}}, /tikz/la1/.style={label={[rotate=90,anchor=west,inner
sep=1pt,label={right:\mathscr{F}([0,\gamma])}]:\in}}, /tikz/la2/.style={label={[rotate=90,anchor=west,inner
sep=1pt,label={right:[0,\gamma]}]:\in}}, /tikz/la3/.style={label={[rotate=90,anchor=west,inner
sep=1pt,label={right:\mathscr{F}\left(\eta^{-1}(\max
    \eta(S))\right)}]:\in}}, /tikz/la4/.style={label={[rotate=90, anchor=west,inner
sep=1pt,label={right:X}]:\in}}]
|[la]| S \arrow[r,"\eta"]  &[-1.5em] 
|[la1]| \eta(S) \arrow[r,"\max"] &[-1em]
|[la2]| \max \eta(S) \arrow[r,"\eta^{-1}"] &[-1em]
|[la3]| \eta^{-1}(\max \eta(S)) \arrow[r, "g_{\max \eta(S)}"]  &
|[la4]| f(S)
\end{tikzcd}
\end{center}

The resulting selection $f=\bigvee_{\alpha\leq \gamma} g_\alpha$ is
continuous in the following special case.

\begin{theorem}
  \label{theorem-Min-Sel-v35:1}
  Let $\eta:X\to [0,\gamma]$ be an ordinal decomposition of a space
  $X$, and $g_\alpha\in \sel[\mathscr{F}(\eta^{-1}(\alpha))]$,
  $\alpha\leq \gamma$. If for each limit ordinal $\lambda\leq\gamma$,
  the selection $g_\lambda$ is $q_\lambda$-minimal, where
  $\eta^{-1}(\lambda)$ is clopen modulo the point
  $q_\lambda\in \eta^{-1}(\lambda)$, then the selection
  $f=\bigvee_{\alpha\leq\gamma}g_\alpha$ defined as in
  \eqref{eq:Min-Sel-v35:3} is continuous.
\end{theorem}

\begin{proof}
  For each $\beta\leq \gamma$, set
  $f_\beta=f\uhr \mathscr{F}\left(\eta^{-1}([0,\beta])\right)
  =\bigvee_{\alpha\leq \beta}g_\alpha$.  Next, take
  $\lambda\leq \gamma$, and assume that $f_\alpha$ is continuous for
  every $\alpha<\lambda$. If $\lambda=\alpha+1$, then the sets
  $\eta^{-1}([0,\alpha])$ and $\eta^{-1}(\lambda)$ are clopen,
  therefore $f_\lambda$ is continuous because so are $f_\alpha$ and
  $g_\lambda=g_{\alpha+1}$. Suppose that $\lambda$ is a limit ordinal,
  and set $Y_\lambda=\eta^{-1}([0,\lambda))\cup\{q_\lambda\}$,
  $Z_\lambda=\eta^{-1}(\lambda)$ and
  $X_\lambda=\eta^{-1}([0,\lambda])$. Then $Y_\lambda$ is also a
  closed set because $Z_\lambda$ is clopen modulo the point
  $q_\lambda$. A crucial role in this proof will play the selection
  ${h_\lambda=f_\lambda\uhr \mathscr{F}(Y_\lambda)}$. Namely, since
  $f_\lambda=f\uhr \mathscr{F}(X_\lambda)$, it follows from
  \eqref{eq:Min-Sel-v35:3} that
  \begin{equation}
    \label{eq:Ord-Dec-vgg:1}
      f_\lambda(S)=
  \begin{cases}
    h_\lambda(S) &\text{if $S\in \mathscr{F}(Y_\lambda)$,\ \ and} \\
    g_\lambda(S\cap Z_\lambda) &\text{if
      $S\cap Z_\lambda\in \mathscr{F}(Z_\lambda)$.}
  \end{cases}
  \end{equation}
  The proof now consists of showing that $h_\lambda$ is continuous, so
  let us take $S\in \mathscr{F}(Y_\lambda)$. If $q_\lambda\notin S$,
  then $\eta(S)\subset [0,\alpha]$ for some $\alpha<\lambda$ because
  $\eta$ is a closed map. Since $\eta^{-1}([0,\alpha])$ is a clopen
  set and $f_\alpha=h_\lambda\uhr \mathscr{F}(\eta^{-1}([0,\alpha]))$
  is continuous, $h_\lambda$ is also continuous at $S$. Otherwise, if
  $q_\lambda\in S$, we have that $S\cap Z_\lambda=\{q_\lambda\}$ and
  by \eqref{eq:Ord-Dec-vgg:1} we obtain that
  $h_\lambda(S)=f_\lambda(S)=q_\lambda$. Now, take an open set
  $U\subset Y_\lambda$ with $q_\lambda\in U$. Since
  $\eta^{-1}(\lambda)$ is clopen modulo the point $q_\lambda$, the set
  $V=U\cup \eta^{-1}(\lambda)$ is open in $X$. Hence,
  $\eta^{-1}((\alpha,\lambda])\subset V$ for some $\alpha<\lambda$
  because the map $\eta$ is closed.  Thus,
  $U_\alpha=\eta^{-1}((\alpha,\lambda])\cap Y_\lambda$ is open in
  $Y_\lambda$ such that $S\in \langle U_\alpha,Y_\lambda\rangle$ and
  $h_\lambda(T)\in U$ for every
  $T\in \langle U_\alpha,Y_\lambda\rangle$, see
  \eqref{eq:Min-Sel-v35:6} and \eqref{eq:Min-Sel-v35:3}. This
  completes the verification that $h_\lambda$ is continuous, and as is
  evident from this verification, $h_\lambda$ is in fact a
  $q_\lambda$-maximal selection for $\mathscr{F}(Y_\lambda)$. Since
  $g_\lambda$ is a $q_\lambda$-minimal selection for
  $\mathscr{F}(Z_\lambda)$, according to \eqref{eq:Ord-Dec-vgg:1} and
  \cite[Lemma 6.4]{garcia-ferreira-gutev-nogura-sanchis-tomita:99}
  (see also \cite[Proposition 3.5]{gutev-nogura:08b}), this implies
  continuity of $f_\lambda$. Thus, by transfinite induction,
  $f=f_\gamma$ is continuous as well.
\end{proof}

In contrast to \eqref{eq:Min-Sel-v35:2}, the construction in
\eqref{eq:Min-Sel-v35:4} is valid for any nonempty subset of the
ordinal space $[0,\gamma]$ being equivalent to the property that
$[0,\gamma]$ is well ordered. Motivated by this, we shall say that a
subset $\mathscr{D}\subset \Delta(X)$ is a \emph{quasi-ordinal
  decomposition} of a space $X$ if $\mathscr{D}$ is a decomposition of
$X$ such that the decomposition space $\mathscr{D}$ has a coarser
topology with respect to which it is a compact ordinal space
$[0,\gamma]$. Thus, identifying $\mathscr{D}$ with the set
$[0,\gamma]$, a quasi-ordinal decomposition $\mathscr{D}$ of $X$ is
identical to a continuous map $\eta:X\to [0,\gamma]$ with
${\mathscr{D}=\left\{\eta^{-1}(\alpha): \alpha\leq
  \gamma\right\}\subset \Delta(X)}$. In this setting, following
\eqref{eq:Min-Sel-v35:4}, for a quasi-ordinal decomposition
$\eta:X\to [0,\gamma]$ of a space $X$ and selections
$g_\alpha\in \sel[\mathscr{F}(\eta^{-1}(\alpha))]$,
$\alpha\leq \gamma$, we will associate the selection
$f=\bigwedge_{\alpha\leq \gamma} g_\alpha$ for $\mathscr{F}(X)$
defined by
\begin{equation}
  \label{eq:Min-Sel-v35:5}
  f(S)=g_{\min \eta(S)}\left(\eta^{-1}(\min \eta(S))\right)\quad
  \text{for every $S\in 
    \mathscr{F}(X)$.} 
\end{equation}

We now have the following result which is complementary to Theorem
\ref{theorem-Min-Sel-v35:1}.   

\begin{theorem}
  \label{theorem-Min-Sel-v35:2}
  Let $\eta:X\to [0,\gamma]$ be a quasi-ordinal decomposition of a
  space $X$, and $g_\alpha\in \sel[\mathscr{F}(\eta^{-1}(\alpha))]$,
  $\alpha\leq \gamma$. If for each limit ordinal $\lambda\leq\gamma$,
  the selection $g_\lambda$ is $q_\lambda$-maximal, where
  $\eta^{-1}(\lambda)$ is clopen modulo the point
  $q_\lambda\in \eta^{-1}(\lambda)$, then the selection
  $f=\bigwedge_{\alpha\leq\gamma}g_\alpha$ defined as in
  \eqref{eq:Min-Sel-v35:5} is continuous.
\end{theorem}

\begin{proof}
  As before, set
  $f_\beta=f\uhr \mathscr{F}\left(\eta^{-1}([0,\beta])\right)
  =\bigwedge_{\alpha\leq \beta}g_\alpha$ for each $\beta\leq \gamma$.
  Next, take $\lambda\leq \gamma$, and assume that $f_\alpha$ is
  continuous for every $\alpha<\lambda$. If $\lambda=\alpha+1$, as in
  the proof of Theorem \ref{theorem-Min-Sel-v35:1}, $f_\lambda$ is
  also continuous. Suppose that $\lambda$ is a limit ordinal, and set
  $Y_\lambda=\eta^{-1}([0,\lambda))\cup\{q_\lambda\}$,
  $Z_\lambda=\eta^{-1}(\lambda)$ and
  $X_\lambda=\eta^{-1}([0,\lambda])$. Since $Z_\lambda$ is clopen
  modulo the point $q_\lambda$, the set $Y_\lambda$ is closed, and the
  proof now consists of showing that
  $h_\lambda=f_\lambda\uhr \mathscr{F}(Y_\lambda)$ is a
  $q_\lambda$-minimal selection for
  $\mathscr{F}(Y_\lambda)$. Evidently, by \eqref{eq:Min-Sel-v35:5},
  $h_\lambda(S)=q_\lambda$ precisely when $S=\{q_\lambda\}$. Moreover,
  $h_\lambda$ is continuous at $\{q_\lambda\}$ because each selection
  is continuous at the singletons. Finally, take
  $S\in \mathscr{F}(Y_\lambda)$ with $S\neq \{q_\lambda\}$, and let
  $\alpha<\lambda$ be such that
  $S\cap \eta^{-1}(\alpha)\neq \emptyset$. Then by
  \eqref{eq:Min-Sel-v35:5}, the clopen set $H=\eta^{-1}([0,\alpha])$
  is such that $h_\lambda(S)=f_\alpha(S\cap H)$. This implies that
  $h_\lambda$ is continuous at $S$ because $f_\alpha$ is continuous at
  $S\cap H$. Thus, precisely as this was done in the proof of Theorem
  \ref{theorem-Min-Sel-v35:1}, $f_\lambda$ is continuous because
  \[
    f_\lambda(S)=
    \begin{cases}
      h_\lambda(S\cap Y_\lambda) &\text{if $S\cap Y_\lambda\in
        \mathscr{F}(Y_\lambda)$, and} \\ 
      g_\lambda(S) &\text{if $S\in \mathscr{F}(Z_\lambda)$.}
    \end{cases}\qedhere
  \]
\end{proof}

Theorems \ref{theorem-Min-Sel-v35:1} and \ref{theorem-Min-Sel-v35:2}
have an almost identical proof based on \cite[Lemma
6.4]{garcia-ferreira-gutev-nogura-sanchis-tomita:99}. Regarding this
lemma, for a set $H\subset X$ clopen modulo a point $q\in H$, define a
map $\varphi:\langle H,X\rangle\to \mathscr{F}(H)$ by
$\varphi(S)=(S\cap H)\cup\{q\}$ for every $S\in \langle H, X\rangle$,
see \eqref{eq:Min-Sel-v35:6}.  Then a very simple argument, given in
\cite[Proposition
6.1]{garcia-ferreira-gutev-nogura-sanchis-tomita:99}, shows that
$\varphi$ is Vietoris continuous. This map is instrumental to glue
selections at a point $q\in X$, as done in \cite[Lemma
6.4]{garcia-ferreira-gutev-nogura-sanchis-tomita:99}.\medskip

In the setting of Theorem \ref{theorem-Min-Sel-v35:1}, if
$\eta^{-1}(\gamma)=\{p\}$ for some point $p\in X$, then the resulting
selection $f=\bigvee_{\alpha\leq\gamma}g_\alpha$ is
$p$-maximal. Similarly, the selection
$f=\bigwedge_{\alpha\leq\gamma}g_\alpha$ in Theorem
\ref{theorem-Min-Sel-v35:2} is $p$-minimal whenever
$\eta^{-1}(\gamma)=\{p\}$.  Such decompositions are not as arbitrary
as it might seem at first glance.

\begin{proposition}
  \label{proposition-Ordinal-Decomp-2nd:1}
  For a space $X$, a non-isolated point $p\in X$ and a decomposition
  $\mathscr{D}\subset \Delta(X)$ with $\{p\}\in \mathscr{D}$, the
  following holds.
  \begin{enumerate}[itemsep=2pt, label=\upshape{(\roman*)}]
  \item\label{item:Ordinal-Decomp-2nd:1} If $\mathscr{D}$ is
    quasi-ordinal, then the singleton $\{p\}$ is an intersection of at
    most $|\mathscr{D}|$ many clopen sets.
  \item\label{item:Ordinal-Decomp-2nd:2} If $\mathscr{D}$ is ordinal,
    then $p$ has a clopen local base of cardinality at most
    $|\mathscr{D}|$.
  \end{enumerate}
\end{proposition}

\begin{proof}
  If $\mathscr{D}$ is as in \ref{item:Ordinal-Decomp-2nd:1}, then it
  is identical to a continuous map $\eta:X\to [0,\gamma]$ such that
  $\mathscr{D}=\left\{\eta^{-1}(\alpha): \alpha\leq \gamma\right\}$
  and $\eta^{-1}(\lambda)=\{p\}$ for some limit ordinal
  $\lambda\leq\gamma$. Moreover, in case of
  \ref{item:Ordinal-Decomp-2nd:2}, the map $\eta$ is also closed. The
  property now follows from the fact that
  $\{p\}=\bigcap_{\alpha<\lambda}(\alpha,\lambda]$.
\end{proof}

Complementary to Proposition \ref{proposition-Ordinal-Decomp-2nd:1} is
the following simple observation.

\begin{proposition}
  \label{proposition-Min-Sel-v35:2}
  For a space $X$ and a non-isolated point $p\in X$, the following
  holds.
  \begin{enumerate}[label=\upshape{(\roman*)}]
  \item\label{item:Ordinal-Decomp-2nd:3} If the singleton $\{p\}$ is a
    countable intersection of clopen sets, then $X$ has a
    quasi-ordinal decomposition $\eta:X\to [0,\omega]$ with
    $\eta^{-1}(\omega)=\{p\}$.
  \item\label{item:Ordinal-Decomp-2nd:4} If $X$ has a countable clopen
    base at $p$, then it also has an ordinal decomposition
    $\eta:X\to [0,\omega]$ with $\eta^{-1}(\omega)=\{p\}$.
  \end{enumerate}
\end{proposition}

\begin{proof}
  The condition in \ref{item:Ordinal-Decomp-2nd:3} implies the
  existence of a family of clopen sets $U_n$, $n<\omega$, such that
  $U_0=X$, $U_{n+1}\subsetneq U_n$, $n<\omega$, and
  $\bigcap_{n<\omega}U_n=\{p\}$.  This is equivalent to the existence
  of a continuous surjective map $\eta:X\to [0,\omega]$ with
  $\eta^{-1}(\omega)=\{p\}$, the relationship is simply given by
  $U_n=\eta^{-1}([n,\omega])$, $n<\omega$. Moreover, $\eta$ is an
  ordinal decomposition precisely when the sets $U_n$, $n<\omega$,
  form a (clopen) base at $p\in X$, i.e.\ in the case of
  \ref{item:Ordinal-Decomp-2nd:4}.
\end{proof}

Proposition \ref{proposition-Min-Sel-v35:2} implies that the following
construction in the proof of \cite[Theorem 1.4]{gutev-nogura:00b} is a
special case of Theorem \ref{theorem-Min-Sel-v35:1}.

\begin{corollary}
  \label{corollary-Min-Sel-v35:1}
  Let $X$ be a space with $\sel[\mathscr{F}(X)]\neq \emptyset$, and
  $p\in X$ be a point which has a countable clopen base. Then
  $\mathscr{F}(X)$ has a $p$-maximal selection. 
\end{corollary}

\begin{proof}
  The trivial case is when $p\in X$ is an isolated point. Otherwise,
  the property follows from Theorem \ref{theorem-Min-Sel-v35:1} and
  Proposition \ref{proposition-Min-Sel-v35:2}. 
\end{proof}

Similarly, Theorem \ref{theorem-Min-Sel-v35:2} contains the following
construction of $p$-minimal selections in \cite[Proposition
3.6]{gutev-nogura:08b}, see also \cite[Lemma
6.3]{garcia-ferreira-gutev-nogura-sanchis-tomita:99} and
\cite[Proposition 4.2]{Gutev2022a}.

\begin{corollary}
  \label{corollary-Min-Sel-v35:2}
  Let $X$ be a space with $\sel[\mathscr{F}(X)]\neq \emptyset$, and
  $p\in X$ be such that $\{p\}$ is a countable intersection of clopen
  subsets of $X$. Then $\mathscr{F}(X)$ has a $p$-minimal
  selection. 
\end{corollary}

\section{A Selection Relation and Sets Clopen Modulo a Point}
\label{sec:select-relat-sets}

The hyperspace selection problem for connected spaces was resolved by
Michael in \cite{michael:51}. It was based on a natural order-like
relation $\leq_\sigma$ on $X$ \cite[Definition 7.1]{michael:51}
defined for $x,y\in X$ by $x\leq_\sigma y$ iff $\sigma(\{x,y\})=x$,
where $\sigma$ is a selection for the subset
$\left\{\{x,y\}: x,y\in X\right\}\subset
\mathscr{F}(X)$. Subsequently, dealing with a hyperspace selection
problem for metrizable spaces, the relation $\leq_\sigma$ was extended
in \cite{zbMATH06882273} to a relation $\leqt_\varphi$ on
$\mathscr{F}(X)$ with respect to a map
$\varphi:\mathscr{F}(X)\to \mathscr{F}(X)$. The definition of
$\leqt_\varphi$ is very similar, namely for $S,T\in \mathscr{F}(X)$ we
write that $S\leqt_\varphi T$ whenever
$f(S\cup T)\cap S\neq \emptyset$. Evidently, for $p\in X$, the
condition $\{p\}\leqt_\varphi T$ means that
$p\in \varphi(T\cup\{p\})$. In this section, we consider the special
case of the relation $\leqt_\varphi$ when
$\varphi=f:\mathscr{F}(X)\to X$ is a selection. In this case, the
corresponding relation $\leqt_f$ is between the points $p\in X$ and
the closed subsets $A\subset X$, and is defined by
\begin{equation}
  \label{eq:Min-Sel-v27:1}
  p\leqt_f A\quad \text{if $f(A\cup\{p\})=p$.}
\end{equation}
Moreover, for convenience, we will write that $p\lhd_f A$ if $p\leqt_f A$
and $p\notin A$.\medskip

In what follows, we will use $\mathscr{T}(X)$ to denote the topology
of a space $X$. Also, the \emph{complement} $X\setminus A$ of a set
$A\subset X$ will be denoted by $A^\cmpl$, i.e.\
$A^\cmpl=X\setminus A$. Thus,
$\mathscr{F}(X)=\left\{V^\cmpl: V\in \mathscr{T}(X)\ \text{and}\ V\neq
  X\right\}$. In these terms, to each $V\in \mathscr{T}(X)$ we will
associate the following sets generated by a selection
$f:\mathscr{F}(X)\to X$.
\begin{equation}
  \label{eq:Min-Sel-v27:2}
  \langle V\rangle_{\leqt_f}=\left\{x\in X: x\lhd_f
    V^\cmpl\right\}\quad
  \text{and}\quad [V]_{\leqt_f}=\left\{x\in X: x\leqt_f
    V^\cmpl\right\}. 
\end{equation}
For a continuous selection $f\in \sel[\mathscr{F}(X)]$, the set
$\langle V\rangle_{\leqt_f}$ can be regarded as the $f$-``interior''
of $V\in \mathscr{T}(X)$, and the set $[V]_{\leqt_f}$\,---\,as the
$f$-``closure'' of $\langle V\rangle_{\leqt_f}$.

\begin{proposition}
  \label{proposition-Min-Sel-v26:1}
  If $f\in\sel[\mathscr{F}(X)]$ and $V\in \mathscr{T}(X)$, then
  $\langle V\rangle_{\leqt_f}\in \mathscr{T}(X)$ and
  $[V]_{\leqt_f}\in \Delta(X)$.  Moreover, if $V\neq X$, then
  $[V]_{\leqt_f}=\langle V\rangle_{\leqt_f}\cup
  \left\{f\left(V^\cmpl\right)\right\}$ and, therefore,
  $[V]_{\leqt_f}$ is clopen modulo the point $f\left(V^\cmpl\right)$.
\end{proposition}

\begin{proof}
  Since the map
  $X\ni x\longrightarrow f_V(x)=f\left(V^\cmpl\cup\{x\}\right)\in X$
  is continuous, the set $[V]_{\leqt_f}=\{x\in X: f_V(x)=x\}$ is
  closed. Moreover, for the same reason,
  $\langle V\rangle_{\leqt_f}=f_V^{-1}(V)$ is open because $f$ is a
  selection for $\mathscr{F}(X)$. If $V=X$, then
  $[X]_{\leqt_f}=\langle X\rangle_{\leqt_f}=X$ because
  $f\left(X^\cmpl\cup \{x\}\right)=f(\emptyset\cup\{x\})=f(\{x\})=x$
  for every $x\in X$.  Otherwise, if $V\neq X$, then
  $f\left(V^\cmpl\cup\{x\}\right)=f\left(V^\cmpl\right)$ for every
  $x\in V^\cmpl$. Accordingly,
  ${f\left(V^\cmpl\right)\in [V]_{\leqt_f}}$ and
  $\langle V\rangle_{\leqt_f}= [V]_{\leqt_f}\cap V=
  [V]_{\leqt_f}\setminus\left\{f\left(V^\cmpl\right)\right\}$.
\end{proof}

Here, we will be mainly interested in the special case when
$f\in \sel[\mathscr{F}(X)]$ is a $p$-maximal selection for some point
$p\in X$ and $p\in V\in \mathscr{T}(X)$. In this case,
$p\lhd_f V^\cmpl$ and, accordingly,
$p\in \langle V\rangle_{\leqt_f}\subset V$. This implies the following
very simple proof of \cite[Lemma
3.4]{garcia-ferreira-gutev-nogura-sanchis-tomita:99}.

\begin{corollary}
  \label{corollary-Min-Sel-v27:1}
  Let $V\in \mathscr{T}(X)$ and $f\in \sel[\mathscr{F}(X)]$ be a
  $p$-maximal selection for some point $p\in V$. If\/
  $W=V\setminus \{q\}$ for some point
  $q\in \langle V\rangle_{\leqt_f}$ with $q\neq p$, then the set
  $[W]_{\leqt_f}$ is clopen modulo the point $f\left(W^\cmpl\right)=q$
  and
  $ p\in \langle W\rangle_{\leqt_f}\subsetneq [W]_{\leqt_f}\subset V$.
\end{corollary}

\begin{proof}
  Since $W\neq X$ and $q\in \langle V\rangle_{\leqt_f}$, by
  Proposition \ref{proposition-Min-Sel-v26:1}, the set $[W]_{\leqt_f}$
  is clopen modulo the point
  $f\left(W^\cmpl\right)=f\left(V^\cmpl\cup \{q\}\right)=q$ and
  $[W]_{\leqt_f}=\langle
  W\rangle_{\leqt_f}\cup\left\{f\left(W^\cmpl\right)\right\}$. Accordingly,
  $[W]_{\leqt_f}=\langle W\rangle_{\leqt_f}\cup\{q\}\subset
  W\cup\{q\}=V$.
\end{proof}

According to Corollary \ref{corollary-Min-Sel-v27:1}, each selection
pointwise-maximal space is regular. In fact, it further implies that
such a space $X$ is also \emph{zero-dimensional}, i.e.\ that $X$ has a
base of clopen sets.

\begin{corollary}
  \label{corollary-Min-Sel-v32:1}
  Each selection pointwise-maximal space $X$ is zero-dimensional.
\end{corollary}

\begin{proof}
  Let $p\in V\in \mathscr{T}(X)$, where $V$ is not clopen. Since
  $\mathscr{F}(X)$ has a $p$-maximal selection, by Corollary
  \ref{corollary-Min-Sel-v27:1}, $p\in H_1\subset V$ for some set
  $H_1$ which is clopen modulo a point $q_1\in H_1$ with $q_1\neq
  p$. Then $H_1\neq V$, and there exists a point
  $q_2\in V\setminus H_1$. Hence, for the same reason,
  $q_1\in H_2\subset V\setminus\{p\}$ for some set $H_2$ which is
  clopen modulo the point $q_2$. Accordingly, $U=H_1\setminus H_2$ is
  a clopen set with $p\in U\subset V$.
\end{proof}

Corollary \ref{corollary-Min-Sel-v32:1} gives a very simple proof of
one of the implications in \cite[Theorem 1.4]{gutev-nogura:00b}, see
also \cite[Theorem 4.3]{gutev-nogura:03a}. Furthermore, the proof of
this fact in \cite{gutev-nogura:00b} is not direct, but is based on
\cite[Theorem 1.5]{gutev-nogura:00b} that every selection
pointwise-maximal space $X$ is \emph{totally disconnected}, i.e.\ that
each singleton of $X$ is an intersection of clopen sets.  The proof in
\cite{gutev-nogura:00b} is also based on maximal chains in the
Vietoris hyperspace. The interested reader is referred to
\cite{bertacchi-costantini:98,costantini-gutev:99,gutev-nogura:00a},
where such chains have played a similar role. \medskip

A space $X$ is \emph{zero-dimensional} at $p\in X$ if it has a clopen
local base at $p$. Similarly, we say that $X$ is \emph{totally
  disconnected} at a point $p\in X$ if the singleton $\{p\}$ is an
intersection of clopen subsets of $X$. Sets clopen modulo a point are
also very efficient to obtain the following result which is
complementary to Corollary \ref{corollary-Min-Sel-v32:1}. It gives a
further simplification of the proof of the aforementioned implication
in \cite[Theorem 1.4]{gutev-nogura:00b}.

\begin{corollary}
  \label{corollary-Min-Sel-v18:1}
  Let $X$ be a space which is totally disconnected at a point $p\in X$
  and has a $p$-maximal selection $f\in \sel[\mathscr{F}(X)]$. Then
  $X$ is zero-dimensional at $p$.
\end{corollary}

\begin{proof}
  Assume that $p$ is not isolated and $V\in \mathscr{T}(X)$
  with $p\in V$. Then there exists a point
  $q\in \langle V\rangle_{\leqt_f}$ with $q\neq p$. Hence, by
  Corollary \ref{corollary-Min-Sel-v27:1}, $V$ contains a
  set $H$ which is clopen modulo the point $q$ and $p\in
  H$. Accordingly, $q\in W\subset X\setminus\{p\}$ for
  some clopen set $W\subset X$. Thus, $U=H\setminus W$ is a
  clopen set with $p\in U\subset V$.
\end{proof}

\section{Cut Points and First Countability} 
\label{sec:cut-points-first}

A point $p\in X$ of a connected space $X$ is \emph{cut} if
$X\setminus\{p\}$ is not connected, equivalently if
$X\setminus \{p\}=U\cup V$ for some subsets $U, V\subset X$ with
$\overline{U}\cap \overline{V}=\{p\}$. Extending this interpretation
to an arbitrary space $X$, a point $p\in X$ is called \emph{cut}
\cite{gutev-nogura:03a} (see also \cite{gutev:00e,gutev-nogura:00d})
if $X\setminus\{p\}=U\cup V$ and $\overline{U}\cap \overline{V}=\{p\}$
for some subsets $U,V\subset X$. Cut points were also introduced in
\cite{Dow2008}, where they were called \emph{tie-points}. The
prototype of such points can be traced back to the so called butterfly
points.  A point $p\in X$ is called \emph{butterfly} (or, a
\emph{b-point}) \cite{Sapirovskii1975} if
$\overline{F\setminus\{p\}}\cap \overline{G\setminus\{p\}}=\{p\}$ for
some closed sets $F,G\subset X$. A point $p\in X$ is butterfly
precisely when it is a cut point of some closed subset of $X$. In
fact, cut points (equivalently, tie-points) are often called
\emph{strong butterfly points}. \medskip

For a cut point $p\in X$, for convenience, we say that a pair $(U,V)$
of subsets $U,V\subset X$ is a \emph{$p$-cut} of $X$ if
$X\setminus\{p\}=U\cup V$ and $\{p\}=\overline{U}\cap
\overline{V}$. Evidently, for a $p$-cut $(U,V)$ of $X$, both sets
$\overline{U}$ and $\overline{V}$ are clopen modulo the point
$p$. Furthermore, in this case, $\overline{U}$ and $\overline{V}$ are
almost clopen. This interpretation is very useful to simplify the
proof of the implication (b)\,$\implies$\,(a) in \cite[Theorem
3.1]{gutev-nogura:03a}. In fact, we have a direct simple proof of the
following part of this implication.

\begin{proposition}
  \label{proposition-Ord-Dec-vgg:1}
  If $X$ is a space which has a $p$-maximal selection
  $f\in\sel[\mathscr{F}(X)]$ for some cut point $p\in X$, then it is
  zero-dimensional at $p$.
\end{proposition}

\begin{proof}
  Take a $p$-cut $(X_0,X_1)$ of $X$ and $V\in \mathscr{T}(X)$ with
  $p\in V$. Since $p$ is a non-isolated point, by Proposition
  \ref{proposition-Min-Sel-v26:1}, for each $i=0,1$ there is a point
  $q_i\in \langle V\rangle_{\leqt_f}\cap X_i$. Next,
  setting $W_i=V\setminus\{q_i\}$, it follows from Corollary
  \ref{corollary-Min-Sel-v27:1} that $[W_i]_{\leqt_f}$ is a set clopen
  modulo the point $q_i$ such that
  $p\in \langle W_i\rangle_{\leqt_f}\subset [W_i]_{\leqt_f}\subset
  V$. Finally, set $U_i=[W_i]_{\leqt_f}\cup X_{i}$, $i=0,1$, which is
  a clopen set because $U_i= [W_i]_{\leqt_f}\cup
  \overline{X_{i}}$. Then $U=U_0\cap U_1$ is a clopen subset of $X$
  such that $p\in U\subset V$ because
  $U_0\cap U_1\subset [W_0]_{\leqt_f}\cup [W_1]_{\leqt_f}\subset V$.
\end{proof}

The original proof of Proposition \ref{proposition-Ord-Dec-vgg:1} in
\cite{gutev-nogura:03a} is not direct, and somewhat technical. First,
it was shown that $X$ is totally disconnected at $p$, and next it was
applied Corollary \ref{corollary-Min-Sel-v18:1} (i.e.\ \cite[Theorem
1.4]{gutev-nogura:00b}). The remaining part of the same implication in
\cite[Theorem 3.1]{gutev-nogura:03a} that $X$ is also first countable
at $p\in X$ was based on a property of a cardinal invariant stated in
\cite[Theorem 4.1]{garcia-ferreira-gutev-nogura-sanchis-tomita:99} and
\cite[Theorem 2.1]{gutev-nogura:03a}. Sets clopen modulo a point are
also very efficient to give a simple direct proof of this fact. The
main idea of this proof is stated in the following proposition.

\begin{proposition}
  \label{proposition-Min-Sel-v29:1}
  Let $f\in\sel[\mathscr{F}(X)]$ be a $p$-maximal selection for some
  cut point $p\in X$, and $(X_0,X_1)$ be a $p$-cut of $X$. Suppose
  that $\{U_n:n<\omega\}\subset \mathscr{T}(X)$ is a sequence of open
  sets such that $p\in U_{n+1}\subset \langle U_{n}\rangle_{\leqt_f}$,
  $n<\omega$, and $f\left(U_{2k+i}^\cmpl\right)\in X_i$ for every
  $k<\omega$ and $i=0,1$. Then $\{U_n:n<\omega\}$ is a local base at
  $p$.
\end{proposition}

\begin{proof}
  By condition, each complement $Y_n=U_n^\cmpl$, $n<\omega$, is
  nonempty because it contains a point of $X_0\cup X_1$. Since these
  complements form an increasing sequence in $\mathscr{F}(X)$, they
  are $\tau_V$-convergent to $Y=\overline{\bigcup_{n<\omega}Y_n}$ and,
  accordingly, $f(Y)=\lim_{n\to\infty}f(Y_n)$. However,
  $f(Y_{2k+i})=f\left(U_{2k+i}^\cmpl\right)\in X_i$ for $k<\omega$ and
  $i=0,1$, which implies that
  $ f(Y)=\lim_{n\to \infty}f(Y_{2n})= \lim_{n\to \infty}f(Y_{2n+1})\in
  \overline{X_0}\cap \overline{X_1}=\{p\}$. Therefore, $p\in Y$. Take
  any point $q\in \bigcap_{n<\omega} U_n$. Then $f(Y_n\cup\{q\})=q$
  because $q\in U_{n+1}\subset \langle U_n\rangle_{\leqt_f}$, see
  \eqref{eq:Min-Sel-v27:1} and \eqref{eq:Min-Sel-v27:2}. Since the
  sequence $Y_n\cup\{q\}$, $n<\omega$, is $\tau_V$-convergent to
  $Y\cup\{q\}$, it follows that $f(Y\cup\{q\})=q$.  Thus, $q=p$
  because $p\in Y$ and $f$ is $p$-maximal, i.e.\
  $\bigcap_{n<\omega} U_n=\{p\}$ and $Y=X$.  Finally, assume that
  $V\subset X$ is an open set such that $p\in V$ and
  $U_n\setminus V\neq \emptyset$, $n<\omega$. Next, for each
  $n<\omega$, take a point $x_n\in U_n\setminus V$. Then
  $f(Y_n\cup\{x_{n+1}\})=f\left(U_n^\cmpl\cup\{x_{n+1}\}\right)=x_{n+1}$
  because $x_{n+1}\in \langle U_{n}\rangle_{\leqt_f}$. Moreover, the
  sequence $Y_n\cup\{y_{n+1}\}$, $n<\omega$, is $\tau_V$-convergent to
  $X=Y$ because $Y=\overline{\bigcup_{n<\omega}Y_n}$. Accordingly,
  $p=f(X)=\lim_{n\to\infty}x_{n+1}$. However, this is impossible
  because $p\in V$ but
  $\lim_{n\to\infty}x_{n+1}\in X\setminus V$.
\end{proof}

The first countability now follows easily from Proposition
\ref{proposition-Min-Sel-v29:1}. It is based on applying twice the
construction in \eqref{eq:Min-Sel-v27:2}. To this end, for $f\in
\sel[\mathscr{F}(X)]$ and $V\in \mathscr{T}(X)$, it will be convenient
to set $\dblr{V}=\langle\langle V\rangle_{\leqt_f}\rangle_{\leqt_f}$.

\begin{corollary}
  \label{corollary-Min-Sel-v25:1}
  Let $X$ be a space which has a $p$-maximal selection
  $f\in\sel[\mathscr{F}(X)]$ for some cut point $p\in X$. Then $X$ is
  first countable at $p$.
\end{corollary}

\begin{proof}
  Let us show how to construct a sequence
  $\{U_n:n<\omega\}\subset \mathscr{T}(X)$ as the one in Proposition
  \ref{proposition-Min-Sel-v29:1} with respect to the point $p$ and a
  $p$-cut $(X_0,X_1)$ of $X$.  To this end, take a $p$-cut $(X_0,X_1)$
  and $U\in \mathscr{T}(X)$ with $p\in U$. Then, according to
  Proposition \ref{proposition-Min-Sel-v26:1},
  ${p\in \dblr{U}\in \mathscr{T}(X)}$ because $f$ is
  $p$-maximal. Hence, there exists a point $q_0\in \dblr{U}\cap
  X_0$. Setting $U_0=\langle U\rangle_{\leqt_f}\setminus\{q_0\}$,
  Corollary \ref{corollary-Min-Sel-v27:1} implies that
  $f\left(U^\cmpl_0\right)=q_0$ and
  ${p\in \langle U_0\rangle_{\leqt_f}\subset U_0\subset \langle
    U\rangle_{\leqt_f}}$. Similarly, taking
  $U_1=\langle U_0\rangle_{\leqt_f} \setminus\{q_1\}$ for some
  $q_1\in \dblr{U_0}\cap X_1$, we get that
  $f\left(U_1^\cmpl\right)=q_1$ and
  $p\in \langle U_1\rangle_{\leqt_f} \subset U_1\subset \langle
  U_0\rangle_{\leqt_f}$.  The construction can be carried on by
  induction. Thus, by Proposition \ref{proposition-Min-Sel-v29:1}, $X$
  is first countable at $p$.
\end{proof}

\section{Ordinal Decompositions Generated by Selections} 
\label{sec:select-pointw-maxim}

A \emph{neighbourhood} of a nonempty closed set $T$ in a space $X$ is
a set which contains $T$ in its interior. Thus, by a
\emph{neighbour\-hood base} at $T$ we mean a collection $\mathscr{H}$
of neighbourhoods of $T$ such that each (open) neighbourhood of $T$
contains a member of $\mathscr{H}$. Moreover, if $\mathscr{H}$
consists of open sets, then it is simply called a \emph{local base} at
$T$.\medskip

The following result is similar to Proposition
\ref{proposition-Min-Sel-v29:1} and offers both a simplification of
the proof and a natural generalisation of \cite[Lemma
4.2]{gutev-nogura:03a}.

\begin{proposition}
  \label{proposition-Min-Sel-v29:2}
  Let $\lambda$ be a limit ordinal, $f\in \sel[\mathscr{F}(X)]$ and
  $U_\alpha\in \mathscr{T}(X)$, $\alpha<\lambda$, be such that
  $\overline{U_{\alpha+1}}\subsetneq \langle
  U_{\alpha}\rangle_{\leqt_f}\subset U_\beta$ for every
  $\beta\leq\alpha<\lambda$. If
  $H_\lambda=\bigcap_{\alpha<\lambda} U_{\alpha}$ and
  $F_\lambda=\overline{H_\lambda^\cmpl}$, then
  $H_\lambda= \left[F_\lambda^\cmpl\right]_{\leqt_f}$ and the family
  $\left\{ U_\alpha:\alpha<\lambda\right\}$ forms a local base at
  $H_\lambda$. In particular, both sets $H_\lambda$ and $F_\lambda$
  are clopen modulo the point $f\left(F_\lambda\right)$.
\end{proposition}

\begin{proof}
  Since
  $H_\lambda\neq U_{\alpha+1}\subset \overline{U_{\alpha+1}}\subset
  \langle U_{\alpha}\rangle_{\leqt_f}\subset U_\alpha$,
  $\alpha<\lambda$, we can always take points
  $q_\alpha\in U_{\alpha+1} \setminus H_\lambda\subset \langle
  U_{\alpha}\rangle_{\leqt_f}$, $\alpha<\lambda$. For such points, it
  follows that
  \begin{equation}
    \label{eq:Ord-Dec-vgg:2}
    \lim_{\alpha<\lambda} q_\alpha=f(F_\lambda)\in H_\lambda.
  \end{equation}
  Indeed, the complements $F_\alpha=U_\alpha^\cmpl$, $\alpha<\lambda$,
  form an increasing net of closed sets with
  $F_\lambda=\overline{\bigcup_{\alpha<\lambda}F_\alpha}$. Moreover,
  by \eqref{eq:Min-Sel-v27:1} and \eqref{eq:Min-Sel-v27:2}, we have
  that $f(F_\alpha\cup\{q_\alpha\})=q_\alpha$ for each
  $\alpha<\lambda$. Furthermore, for each $\alpha<\lambda$ there
  exists $\mu\in [\alpha,\lambda)$ such that $q_\alpha\in F_\mu$
  because $q_\alpha\notin H_\lambda$, so
  $F_\alpha\cup\{q_\alpha\}\subset F_\mu\cup\{q_\mu\}$. Hence, by
  \cite[Proposition 2.1]{Gutev2024a}, the net
  $F_\alpha\cup\{q_\alpha\}$, $\alpha<\lambda$, is $\tau_V$-convergent
  to
  $F_\lambda=
  \overline{\bigcup_{\alpha<\lambda}(F_\alpha\cup\{q_\alpha\})}$.
  Accordingly,
  $f(F_\lambda)=\lim_{\alpha<\lambda} f(F_\alpha\cup\{q_\alpha\})=
  \lim_{\alpha<\lambda} q_\alpha$ and
  $f(F_\lambda)\in
  \bigcap_{\alpha<\lambda}\overline{U_\alpha}=H_\lambda$. This also
  shows that $\{U_\alpha:\alpha<\lambda\}$ forms a local base at
  $H_\lambda$. Namely, if $V\in \mathscr{T}(X)$ is such that
  $H_\lambda\subset V$ and $U_\alpha\setminus V\neq \emptyset$ for
  every $\alpha<\lambda$, then we can take points
  $q_\alpha\in U_{\alpha+1}\setminus V\subset U_{\alpha+1}\setminus
  H_\lambda$, $\alpha<\lambda$. Hence, by \eqref{eq:Ord-Dec-vgg:2},
  $\lim_{\alpha<\lambda} q_\alpha=f(F_\lambda)\in H_\lambda\subset V$,
  but this is impossible because
  $\lim_{\alpha<\lambda} q_\alpha\notin V$. Finally, to see that
  $H_\lambda=\left[F_\lambda^\cmpl\right]_{\leqt_f}$, take a point
  $x\in H_\lambda$. Since
  $x\in\bigcap_{\alpha<\lambda}\langle U_\alpha\rangle$, as before,
  $f(F_\alpha\cup\{x\})=x$ for each $\alpha<\lambda$. Thus,
  $f(F_\lambda\cup\{x\})=\lim_{\alpha<\lambda} f(F_\alpha\cup\{x\})=
  \lim_{\alpha<\lambda} x=x$ because the net $F_\alpha\cup\{x\}$,
  $\alpha<\lambda$, is $\tau_V$-convergent to $F_\lambda\cup\{x\}$.
\end{proof}

The \emph{pseudocharacter} $\psi(p, X)$ of a space $X$ at a point
$p\in X$ is the smallest cardinal number $\kappa$ such that
$\{p\} = \bigcap \mathscr{U}$ for some family
$\mathscr{U}\subset \mathscr{T}(X)$ with $|\mathscr{U}|\leq
\kappa$. The following further property in the setting of selection
pointwise-maximal spaces provides a very simple proof of \cite[Theorem
4.3]{gutev-nogura:03a}.

\begin{proposition}
  \label{proposition-Min-Sel-v32:1}
  Let $X$ be a selection pointwise-maximal space, $p\in X$ be a
  non-isolated point, $f\in \sel[\mathscr{F}(X)]$ be a $p$-maximal
  selection and $\gamma=\psi(p,X)$. Then $X$ has a local  base
  ${\{U_\alpha:\alpha<\gamma\}}$ at the point $p$ such that for every
  $\lambda<\gamma$, 
  \begin{enumerate}[itemsep=2pt, label=\upshape{(\roman*)}]
  \item\label{item:Min-Sel-v30:1}
    $U_{\lambda+1}$ is a clopen set with $U_{\lambda+1}\subsetneq \langle
    U_{\lambda}\rangle_{\leqt_f}$.
  \item\label{item:Min-Sel-v30:2}
    $\bigcap_{\alpha<\lambda} U_\alpha= U_\lambda\cup
    \left\{f\left(U_\lambda^\cmpl\right)\right\}$ and
    $\psi\left(f\left(U_\lambda^\cmpl\right), \bigcap_{\alpha<\lambda}
      U_\alpha\right)\leq \omega$, whenever $\lambda$ is a limit
    ordinal.
  \end{enumerate}
\end{proposition}

\begin{proof}
  Since $\{p\}=\bigcap_{\alpha<\gamma} V_{\alpha+1}$ for some family
  $\{V_{\alpha+1}: \alpha<\gamma\}\subset \mathscr{T}(X)$, by
  Proposition \ref{proposition-Min-Sel-v29:2}, it suffices to subject
  \ref{item:Min-Sel-v30:1} to the additional condition that
  ${U_{\lambda+1}\subset V_{\lambda+1}}$. So, take $\lambda<\gamma$
  and set $U_0=X$. If $p\in U_\lambda$ for some
  $U_\lambda\in \mathscr{T}(X)$, then $p\in \langle U_\lambda\rangle$
  and by Corollary \ref{corollary-Min-Sel-v32:1},
  $p\in U_{\lambda+1}\subsetneq V_{\lambda+1}\cap \langle
  U_\lambda\rangle_{\leqt_f}$ for some clopen set
  $U_{\lambda+1}\subset X$. If $\lambda$ is a limit ordinal and
  $U_\alpha$ is defined for every $\alpha<\lambda$, set
  $H_\lambda=\bigcap_{\alpha<\lambda}U_\alpha$ and
  $F_\lambda=\overline{H_\lambda^\cmpl}$. Then by Proposition
  \ref{proposition-Min-Sel-v29:2},
  $H_\lambda= \left[F_\lambda^\cmpl\right]_{\leqt_f}$ is clopen modulo
  the point $f(F_\lambda)$. Moreover, $f(F_\lambda)\neq p$ because the
  selection $f$ is $p$-maximal. Indeed, according to
  \eqref{eq:Min-Sel-v27:1} and \eqref{eq:Min-Sel-v27:2},
  $f(F_\lambda)=p$ will imply that
  $H_\lambda=\left[F_\lambda^\cmpl\right]_{\leqt_f}=\{p\}$ which is
  impossible because $\lambda<\gamma=\psi(p,X)$. Thus,
  $U_\lambda=H_\lambda\setminus \left\{f(F_\lambda)\right\}$ is an
  open set with $p\in U_\lambda$. The remaining property in
  \ref{item:Min-Sel-v30:2} now follows from the fact that
  $f(F_\lambda)$ is a non-isolated point of $F_\lambda$. If
  $f(F_\lambda)$ is also not isolated in $H_\lambda$, then it is a cut
  point of $X$. Accordingly, by Corollary
  \ref{corollary-Min-Sel-v25:1}, $X$ is first countable at
  $f(F_\lambda)$.
\end{proof}

The \emph{character} $\chi(p, X)$ of a space $X$ at a point $p\in X$
is the smallest cardinal number $\kappa$ such that $X$ has a local
base $\mathscr{B}$ at $p$ with $|\mathscr{B}|\leq \kappa$. The
following property is implicitly present in Proposition
\ref{proposition-Min-Sel-v32:1}, and implies \cite[Corollary
5.3]{gutev-nogura:03a}.

\begin{corollary}
  \label{corollary-Min-Sel-v39:1}
  For a selection pointwise-maximal space $X$, the following holds.
  \begin{enumerate}[itemsep=2pt, label=\upshape{(\roman*)}]
  \item\label{item:Ord-Dec-vgg:1} ${\chi(q,X)=\psi(q,X)}$ for every
    $q\in X$.
  \item\label{item:Ord-Dec-vgg:2} The set
    $\{q\in X: \chi(q,X)\leq\omega\}$ is dense in $X$.
  \end{enumerate}
\end{corollary}

\begin{proof}
  The property in \ref{item:Ord-Dec-vgg:1} is explicitly stated in
  Proposition \ref{proposition-Min-Sel-v32:1}. To see
  \ref{item:Ord-Dec-vgg:2}, we will use that $X$ is zero-dimensional,
  see Corollary \ref{corollary-Min-Sel-v32:1}. Since each clopen
  subset of $X$ is also a selection pointwise-maximal space, by
  \ref{item:Ord-Dec-vgg:1}, it now suffices to show that $X$ contains
  a point $q\in X$ with $\psi(q,X)\leq\omega$. This is trivial when
  $X$ has an isolated point. Otherwise, for a non-isolated point
  $p\in X$, a $p$-maximal selection $f\in \sel[\mathscr{F}(X)]$ and
  $\gamma=\psi(p,X)$, let $\{U_\alpha:\alpha<\gamma\}$ be as in
  Proposition \ref{proposition-Min-Sel-v32:1}. Then taking
  $q=f\left(U_\omega^\cmpl\right)$, it follows from this proposition
  that $\psi(q,X)\leq \omega$.
\end{proof}

For a space $X$, motivated by Corollary \ref{corollary-Min-Sel-v39:1},
we consider the subfamily $\Delta_\omega(X)$ of $\Delta(X)$ consisting
of all $H\in \Delta(X)$ such that $H$ is clopen or 
\begin{equation}
  \label{eq:Min-Sel-v39:1}
  H\ \text{is clopen modulo a point $q\in H$ with
    $\chi(q,H)\leq \omega$}.  
\end{equation}
Evidently, $\Delta_\omega(X)$ contains all singletons of
$X$. Furthermore, for a selection pointwise-maximal space, according
to Corollary \ref{corollary-Min-Sel-v39:1}, $\Delta_\omega(X)$
actually consists of all sets $H\in \Delta(X)$ such that $H$ is clopen
modulo a point $p\in H$ with $\chi(p,H)\leq\omega$.\medskip

A space $X$ is said to be \emph{well-orderable at a point $p\in X$}
\cite{gutev-nogura:03a} if there exists a regular cardinal $\gamma$
and a neighbourhood base
$\{H_{\alpha}:\alpha<\gamma\}\subset \Delta_\omega(X)$ at the point
$p$ such that for every $\lambda<\gamma$,
\begin{enumerate}[itemsep=2pt, label=(\thesection.\arabic*), series=base]
  \addtocounter{enumi}{2}
\item\label{item:Min-Sel-v33:2} $H_{\lambda+1}$ is a clopen set with
  $H_{\lambda+1}\subsetneq H_{\lambda}$.
\item\label{item:Min-Sel-v33:3} If $\lambda$ is a limit ordinal, then
  the family $\{H_{\alpha}:\alpha<\lambda\}$ forms a neighbourhood
  base at $H_{\lambda}$.
\end{enumerate}
In \cite{gutev-nogura:03a}, a neighbourhood base
$\{H_{\alpha}:\alpha<\gamma\}$ with the above properties was called a
\emph{$\gamma$-base} at $p$. Each isolated point $p\in X$ has a
$1$-base, we can simply take $H_0=\{p\}$. Thus, the essential
considerations will be for a non-isolated point $p\in X$.\medskip

Following the idea of selection pointwise-maximal spaces, we will say
that $X$ is \emph{pointwise well-orderable} if it is well-orderable at
each point of $X$. Regarding such spaces, we will clarify the ideas
underlying the following result obtained in \cite[Theorem
4.4]{gutev-nogura:03a}, and significantly simplify its proof.

\begin{theorem}[\cite{gutev-nogura:03a}]
  \label{theorem-point-extreme-complete}
  For a space $X$ with $\sel[\mathscr{F}(X)]\neq \emptyset$, the following
  conditions are equivalent\textup{:} 
  \begin{enumerate}
  \item\label{item:sa_points-v10:1} $X$ is a selection
    pointwise-maximal space.
  \item\label{item:sa_points-v10:2} $X$ is a pointwise well-orderable
    space. 
  \item\label{item:sa_points-v10:3} For each $p\in X$, the space $X$
    has an ordinal decomposition
    ${\mathscr{D}\subset \Delta_\omega(X)}$ with
    $\{p\}\in \mathscr{D}$.
  \end{enumerate}
\end{theorem}

\begin{proof}
  In this proof, $p\in X$ is a fixed non-isolated point.\smallskip
  
  \ref{item:sa_points-v10:1}$\implies$\ref{item:sa_points-v10:2}. Let
  $f\in\sel[\mathscr{F}(X)]$ be a $p$-maximal selection and
  $\{U_\alpha:\alpha<\gamma\}$ be as in Proposition
  \ref{proposition-Min-Sel-v32:1}, where $\gamma=\psi(p,X)$. For
  $\lambda<\gamma$, set
  $H_\lambda=U_\lambda\cup\left\{f\left(U_\lambda^\cmpl\right)\right\}$
  if $\lambda$ is a limit ordinal, and $H_\lambda=U_\lambda$
  otherwise. Then by Propositions \ref{proposition-Min-Sel-v29:2} and
  \ref{proposition-Min-Sel-v32:1}, the family
  $\{H_\alpha:\alpha<\gamma\}$ is a $\gamma$-base at $p$.\smallskip

  \ref{item:sa_points-v10:2}$\implies$\ref{item:sa_points-v10:3}.  Let
  $\{H_{\alpha}:\alpha<\gamma\}$ be a $\gamma$-base at the point $p$
  with $H_0=X$. Next, using \ref{item:Min-Sel-v33:3}, define
  $\eta:X\to [0,\gamma]$ by
  $\eta(x)=\max\{\alpha\leq \gamma: x\in H_\alpha\}$, $x\in X$. Then
  $\eta^{-1}((\alpha,\gamma])=H_{\alpha+1}$ and
  $\eta^{-1}((\alpha,\lambda])= H_{\alpha+1}\setminus H_{\lambda+1}$
  for every $\alpha<\lambda<\gamma$.  Hence, by
  \ref{item:Min-Sel-v33:2}, the map $\eta$ is continuous.  Evidently,
  $\eta^{-1}(\gamma)=\{p\}\in \Delta_\omega(X)$ and for every
  $U\in \mathscr{T}(X)$ with $p\in U$, there exists $\alpha<\gamma$
  such that $\eta^{-1}((\alpha,\gamma])\subset U$. Let
  $\lambda<\gamma$ be a limit ordinal. Then by
  \ref{item:Min-Sel-v33:2} and \ref{item:Min-Sel-v33:3}, $H_\lambda$
  is not an open set. Since $H_{\lambda+1}$ is clopen, it follows from
  \eqref{eq:Min-Sel-v39:1} that $H_\lambda$ is clopen modulo a point
  $q_\lambda\in H_\lambda\setminus H_{\lambda+1}$ with
  $\chi(q_\lambda,H_\lambda)\leq \omega$. Thus,
  $q_\lambda\in \eta^{-1}(\lambda)$ and
  $\eta^{-1}(\lambda)\in \Delta_\omega(X)$. Finally, take an open set
  $U\in \mathscr{T}(X)$ with $\eta^{-1}(\lambda)\subset U$. Since
  $U\cup H_{\lambda+1}$ is open and
  $H_\lambda\subset U\cup H_{\lambda+1}$, it follows from
  \ref{item:Min-Sel-v33:3} that $H_\alpha\subset U\cup H_{\lambda+1}$
  for some $\alpha<\lambda$. Accordingly,
  $\eta^{-1}((\alpha,\lambda])\subset U$ which implies that $\eta$ is
  a closed map as well, so
  $\mathscr{D}=\left\{\eta^{-1}(\alpha):
    \alpha\leq\gamma\right\}\subset \Delta_\omega(X)$ is an ordinal
  decomposition of $X$ with $\{p\}\in
  \mathscr{D}$. \smallskip

  \ref{item:sa_points-v10:3}$\implies$\ref{item:sa_points-v10:1}.  For
  ordinals $\lambda<\gamma$, the topological sum
  $(\lambda,\gamma]\uplus [0,\lambda]$ is actually the ordinal space
  $[0,\gamma]$. Hence, by \ref{item:sa_points-v10:3}, there exists a
  quotient closed map $\eta:X\to [0,\gamma]$ such that
  $\left\{\eta^{-1}(\alpha): \alpha\leq \gamma\right\}\subset
  \Delta_\omega(X)$ and $\eta^{-1}(\gamma)=\{p\}$. Take a limit
  ordinal $\lambda\leq\gamma$. Since
  $\eta^{-1}(\lambda)\in \Delta_\omega(X)$ and is not a clopen set, it
  is clopen modulo a point $q_\lambda\in \eta^{-1}(\lambda)$ with
  $\chi(q_\lambda, \eta^{-1}(\lambda))\leq \omega$, see
  \eqref{eq:Min-Sel-v39:1}. Moreover, by \ref{item:sa_points-v10:3}
  and Proposition \ref{proposition-Ordinal-Decomp-2nd:1}, the space
  $X$ is zero-dimen\-sional. Therefore, by Corollary
  \ref{corollary-Min-Sel-v35:2}, $\eta^{-1}(\lambda)$ has a
  $q_\lambda$-minimal selection
  $g_\lambda\in
  \sel\left[\mathscr{F}\left(\eta^{-1}(\lambda)\right)\right]$ because
  $\sel[\mathscr{F}(X)]\neq \emptyset$. For any other ordinal
  $\alpha\leq \gamma$, take a selection
  $g_\alpha\in \sel\left[\mathscr{F}\left(\eta^{-1}(\alpha)
    \right)\right]$. Then, according to \eqref{eq:Min-Sel-v35:3} and
  Theorem \ref{theorem-Min-Sel-v35:1}, the resulting selection
  $f=\bigvee_{\alpha\leq\gamma}g_\alpha\in \sel[\mathscr{F}(X)]$ is
  $p$-maximal.
\end{proof}

\section{Point-Extreme Selections} 
\label{sec:point-extr-select}

Subspaces of orderable spaces are not necessarily orderable, they are
termed \emph{suborderable}. A space $X$ is \emph{strongly
  zero-dimensional} if its \emph{covering dimension} is zero. A space
$X$ is \emph{ultranormal} if every two disjoint closed subsets are
contained in disjoint clopen subsets; equivalently, if its \emph{large
  inductive dimension} is zero. In the realm of normal spaces, $X$ is
ultranormal if and only if it is strongly zero-dimensional. Each
totally disconnected suborderable space is ultranormal. This was
essentially shown in H. Herrlich \cite[Lemma 1]{MR0185564} and
explicitly stated in S. Purisch \cite[Proposition
2.3]{purisch:77}.\label{page:suborderable-zero} Moreover, each
suborderable space is both countably paracompact
\cite{Ball1954,MR0063646} and collectionwise normal \cite{ MR0093753,
  MR0257985}, see also \cite{engelking:89}.\medskip

Several authors contributed to the following fundamental result about
the orderability of the \emph{{\v C}ech-Stone compactification}
$\beta X$ of a Tychonoff space $X$. Below we state only a partial case
of this result.

\begin{theorem}[\cite{artico-marconi-pelant-rotter-tkachenko:02,
    douwen:90, garcia-ferreira-sanchis:04, glicksber:59,
    mill-wattel:81, miyazaki:01b,
    venkataraman-rajagopalan-soundararajan:72}]
  \label{theorem-Min-Sel-v40:1}\
    \begin{enumerate}[itemsep=1pt, label=\upshape{(\roman*)}]
    \item A compact space $Y$ is orderable if and only if
  $\sel[\mathscr{F}(Y)]\neq \emptyset$. 
\item If $X$ is a Tychonoff space and $\beta X$ is orderable, then $X$
  is pseudocompact.
\item If $X$ is a pseudocompact space with $\sel[\mathscr{F}(X)]\neq
  \emptyset$, then $\sel[\mathscr{F}(\beta X)]\neq \emptyset$. 
\end{enumerate}
\end{theorem}

Based on this result, we will obtain the following consequence of Theorem
\ref{theorem-point-extreme-complete}.

\begin{theorem}
  \label{theorem-Min-Sel-v40:2}
  If $X$ is a pseudocompact selection pointwise-maximal space, then
  $\beta X$ has a $p$-maximal selection
  $f\in \sel[\mathscr{F}(\beta X)]$ for every $p\in X$.
\end{theorem}

To prepare for the proof of Theorem \ref{theorem-Min-Sel-v40:2},
let us recall that a set-valued mapping $\varphi:Y\to \mathscr{F}(Z)$
between spaces $Y$ and $Z$ is \emph{upper semi-continuous}, or
\emph{u.s.c.}, if $\varphi^\#[U]=\{y\in Y: \varphi(y)\subset U\}$ is
open in $Y$ for every open $U\subset Z$. A compact-valued u.s.c.\
mapping $\varphi:Y\to \mathscr{F}(Z)$ is commonly called
\emph{usco}. It has the property that
$\varphi[K]=\bigcup_{y\in K}\varphi(y)$ is a compact subset of $Z$ for
every compact set $K\subset Y$. This implies the following interesting
property of extending closed surjective maps.

\begin{proposition}
  \label{proposition-Min-Sel-v40:1}
  Let $f:X\to Y$ be a closed surjective map, where $X$ is a Tychonoff
  space and $Y$ is a compact space. If $\beta f:\beta X\to Y$ is the
  corresponding extension to the {\v C}ech-Stone compactification of
  $X$, then $(\beta f)^{-1}(y)= \overline{f^{-1}(y)}$ for every $y\in
  Y$. 
\end{proposition}

\begin{proof}
  Since $f$ is closed, the inverse map $f^{-1}:Y\to \mathscr{F}(X)$ is
  u.s.c. Accordingly, the mapping $\varphi:Y\to \mathscr{F}(\beta X)$
  defined by $\varphi(y)=\overline{f^{-1}(y)}$, $y\in Y$, is
  usco. Moreover, $\varphi(y)\subset (\beta f)^{-1}(y)$ for every
  $y\in Y$. Finally, since $X$ is dense in $\beta X$ and $\varphi[Y]$
  is a compact subset of $\beta X$, we get that $\varphi[Y]=\beta X$. 
\end{proof}

We also need the following special case of Proposition
\ref{proposition-Min-Sel-v40:1}.

\begin{proposition}
  \label{proposition-Min-Sel-v40:2}
  If $X$ is a zero-dimensional normal space and
  $H\in \Delta_\omega(X)$, then
  $\overline{H}\in \Delta_\omega(\beta X)$.
\end{proposition}

\begin{proof}
  If $H$ is clopen in $X$, then $\overline{H}$ is clopen in $\beta X$,
  see e.g.\ \cite[Corollary 3.6.5]{engelking:89}. Otherwise, if $H$ is
  not clopen, then it is clopen modulo a point $q\in H$ such that $H$
  has a countable clopen base at $q$. In case $q$ is an isolated point
  of $H$, it follows that $H\setminus\{q\}$ is a clopen subset of
  $X$. Hence, as before, $q$ is an isolated point of
  $\overline{H}=\{q\}\cup \overline{ H\setminus \{q\}}$. Finally,
  assume that $q$ is not isolated in $H$, and take a strictly
  decreasing clopen base $\{U_n: n<\omega\}$ at $q$ in $H$ with
  $U_0=H$. Next, set $U_\omega=\{q\}$ and define a closed continuous
  surjective map $h:H\to [0,\omega]$ by
  $h(x)=\max\{n\leq \omega: x\in U_n\}$, $x\in H$. Since $X$ is a
  normal space, it follows that $\beta H=\overline{H}$, see e.g.\
  \cite[Corollary 3.6.8]{engelking:89}. Hence, $h$ can be extended to
  a continuous map $\beta h:\overline{H}\to [0,\omega]$. Then,
  according to Proposition \ref{proposition-Min-Sel-v40:1},
  $(\beta h)^{-1}(n)=\overline{h^{-1}(n)}$, $n\leq \omega$. This
  implies that $\overline{H}$ is clopen modulo the point $q$ because
  $(\beta h)^{-1}(\omega)=\{q\}$ and
  $(\beta h)^{-1}(n)=\overline{U_n\setminus U_{n+1}}$ is clopen in
  $\beta X$ for every $n<\omega$. Thus,
  $\overline{H}\in \Delta_\omega(X)$ because the family
  $(\beta h)^{-1}([n,\omega])$, $n< \omega$, forms a local base at $q$
  in $\overline{H}$.
\end{proof}

\begin{proof}[Proof of Theorem \ref{theorem-Min-Sel-v40:2}]
  Take a non-isolated point $p\in X$. Since $\mathscr{F}(X)$ has a
  $p$-maximal selection, by Theorem
  \ref{theorem-point-extreme-complete}, $X$ has an ordinal
  decomposition $\eta:X\to [0,\gamma]$ such that
  $\left\{\eta^{-1}(\alpha):\alpha\leq \gamma\right\}\subset
  \Delta_\omega(X)$ and $\eta^{-1}(\gamma)=\{p\}$. Then the
  corresponding extension $\beta\eta:\beta X\to [0,\gamma]$ has the
  same properties with respect to the point $p\in \beta X$. Namely, by
  Proposition \ref{proposition-Min-Sel-v40:1},
  $(\beta\eta)^{-1}(\gamma)=\{p\}$ and
  $(\beta\eta)^{-1}(\alpha)=\overline{\eta^{-1}(\alpha)}$ for every
  $\alpha<\gamma$. However, by Corollary \ref{corollary-Min-Sel-v32:1}
  and Theorem \ref{theorem-Min-Sel-v40:1}, $X$ is a zero-dimensional
  normal space. Hence, by Proposition \ref{proposition-Min-Sel-v40:2},
  $(\beta\eta)^{-1}(\alpha)\in \Delta_\omega(\beta X)$ for every
  $\alpha\leq\gamma$. In fact, for each limit ordinal
  $\lambda\leq\gamma$, the set
  $(\beta\eta)^{-1}(\lambda)=\overline{\eta^{-1}(\lambda)}$ is clopen
  modulo a point $q_\lambda\in \eta^{-1}(\lambda)$ with
  $\chi\left(q_\lambda, (\beta \eta)^{-1}(\lambda)\right)\leq
  \omega$. We can now follow the proof of the implication
  \ref{item:sa_points-v10:3}$\implies$\ref{item:sa_points-v10:1} in
  Theorem \ref{theorem-point-extreme-complete}. Briefly, by Theorem
  \ref{theorem-Min-Sel-v40:1},
  $\sel[\mathscr{F}(\beta X)]\neq \emptyset$. Therefore, by Corollary
  \ref{corollary-Min-Sel-v35:2}, $(\beta \eta)^{-1}(\lambda)$ has a
  $q_\lambda$-minimal selection
  $g_\lambda\in \sel\left[\mathscr{F}\left((\beta
      \eta)^{-1}(\lambda)\right)\right]$ for each limit ordinal
  $\lambda\leq \gamma$. For any other ordinal $\alpha\leq \gamma$,
  take a selection
  $g_\alpha\in \sel\left[\mathscr{F}\left((\beta\eta)^{-1}(\alpha)
    \right)\right]$. Then by \eqref{eq:Min-Sel-v35:3} and Theorem
  \ref{theorem-Min-Sel-v35:1}, the selection
  $f=\bigvee_{\alpha\leq\gamma}g_\alpha\in \sel[\mathscr{F}(\beta X)]$
  is $p$-maximal.
\end{proof}

Theorem \ref{theorem-Min-Sel-v40:2} brings the following natural
question.

\begin{question}
  \label{question-Min-Sel-v40:1}
  Let $X$ be a pseudocompact space which is selection
  pointwise-maximal. Then, is it true that $\beta X$ is also selection
  pointwise-maximal?
\end{question}

Regarding other open questions about the selection problem for
pseudocompact spaces, the interested reader is refereed to
\cite[Problems 3.13 and 3.15]{gutev-2013springer}, see also
\cite[Question 389]{gutev-nogura:06b}.\medskip

Another interesting application is the following property of selection
pointwise-maximal spaces. 

\begin{theorem}
  \label{theorem-sa_points-v9:2}
  Each selection pointwise-maximal space is selection
  pointwise-minimal. 
\end{theorem}

\begin{proof}
  Take a non-isolated point $p\in X$. Since $\mathscr{F}(X)$ has a
  $p$-maximal selection, by Theorem
  \ref{theorem-point-extreme-complete}, $X$ has an ordinal
  decomposition $\eta:X\to [0,\gamma]$ such that
  $\left\{\eta^{-1}(\alpha):\alpha\leq \gamma\right\}\subset
  \Delta_\omega(X)$ and $\eta^{-1}(\gamma)=\{p\}$.  Take a limit
  ordinal $\lambda\leq\gamma$. Then $\eta^{-1}(\lambda)$ is not a
  clopen set and, according to \eqref{eq:Min-Sel-v39:1}, it is clopen
  modulo a point $q_\lambda\in \eta^{-1}(\lambda)$ with
  $\chi\left(q_\lambda, \eta^{-1}(\lambda)\right)\leq \omega$.  Hence,
  by Corollaries \ref{corollary-Min-Sel-v35:1} and
  \ref{corollary-Min-Sel-v32:1}, $\eta^{-1}(\lambda)$ has a
  $q_\lambda$-maximal selection
  $g_\lambda\in \sel\left[\mathscr{F}\left(
      \eta^{-1}(\lambda)\right)\right]$. For any other ordinal
  $\alpha\leq \gamma$, take a selection
  $g_\alpha\in \sel\left[\mathscr{F}\left(\eta^{-1}(\alpha)
    \right)\right]$. Then by \eqref{eq:Min-Sel-v35:5} and Theorem
  \ref{theorem-Min-Sel-v35:2}, the selection
  $f=\bigwedge_{\alpha\leq\gamma}g_\alpha\in \sel[\mathscr{F}(\beta
  X)]$ is $p$-minimal.
\end{proof}

There are simple examples of selection pointwise-minimal spaces which
are not selection pointwise-maximal.

\begin{example}
  \label{example-sa_points-v10:1}
  Identifying the limit points of countably many nontrivial converging
  sequences and equipping the resulting space with the quotient
  topology, we get the so called \emph{sequential fan} $S_\omega$. The
  only non-isolated point of $S_\omega$ is often denoted by
  $\infty$. The sequential fan $S_\omega$ is a non-metrizable
  Fr\'echet-Urysohn space, but $\infty\in S_\omega$ is a
  $G_\delta$-point of $S_\omega$. According to \cite[Corollary
  4.3]{garcia-ferreira-gutev-nogura-sanchis-tomita:99} and Corollary
  \ref{corollary-Min-Sel-v35:2}, $S_\omega$ has a selection
  $f\in \sel[\mathscr{F}(S_\omega)]$ which is $\infty$-minimal. Since
  all other points of $S_\omega$ are isolated, $S_\omega$ is a
  selection pointwise-minimal space. However, $\mathscr{F}(S_\omega)$
  has no $\infty$-maximal selection. This follows from Corollary
  \ref{corollary-Min-Sel-v25:1} because $\infty$ is a cut point of
  $S_\omega$, but $S_\omega$ is not first countable at $\infty$.
\end{example}

Theorems
\ref{theorem-Min-Sel-v35:2} and \ref{theorem-sa_points-v9:2} bring the
following natural question.

\begin{question}
  \label{question-Min-Sel-v40:2}
  Let $X$ be a selection pointwise-minimal space and $p\in X$. Then,
  is it true that $X$ has a quasi-ordinal decomposition
  $\mathscr{D}\subset \Delta_\omega(X)$ with $\{p\}\in \mathscr{D}$?
\end{question}

Regarding Question \ref{question-Min-Sel-v40:2}, let us remark that
each selection pointwise-minimal space is totally disconnected as
pointed out in the Introduction, see \cite[Proposition
4.4]{MR3122363}.  Moreover, the answer to this question is ``Yes'' in
the special case when $p\in X$ is an isolated point, or when there
exists a countable set $S\subset X\setminus\{p\}$ with
$p\in \overline{S}$. In this special case, the singleton $\{p\}$ is a
countable intersection of clopen set which follows from \cite[Theorem
4.4]{garcia-ferreira-gutev-nogura-sanchis-tomita:99} and
\cite[Proposition 5.6]{gutev-nogura:09a}, see also \cite[Remark
3.5]{MR3478342}.\medskip

According to \cite[Proposition 4.1]{Gutev2022a} (see also the proof of
\cite[Theorem 3.1]{gutev-nogura:00d}), if $X$ is a nontrivial
selection pointwise-minimal space, then
$\sel[\mathscr{F}(X\setminus\{p\})]\neq \emptyset$ for every $p\in
X$. Regarding hyperspace selections avoiding points, it was shown in
\cite[Corollary 5.2]{MR3712970} that if
$\sel[\mathscr{F}(Z)]\neq \emptyset$ for each subset $Z\subset X$ with
$|X\setminus Z|\leq 1$, then each connected component of $X$ is
compact. Subsequently, this result was refined in \cite[Theorem
1.1]{Gutev2022a} that a non-degenerate connected space $X$ is compact
and orderable if and only if
$\sel[\mathscr{F}(X\setminus\{p\})]\neq \emptyset$ for every $p\in
X$. Complementary to this result, it was shown in \cite[Proposition
4.5]{Gutev2022a} (the result is credited to J. A. C. Chapital) that if
$\sel[\mathscr{F}(X\setminus\{p\})]\neq \emptyset$ for every $p\in X$,
then $\sel[\mathscr{F}(X)]\neq \emptyset$. Moreover, it was shown in
\cite[Corollary 5.3]{MR3712970} that if
$\sel[\mathscr{F}(Z)]\neq \emptyset$ for every $Z\subset X$ with
$|X\setminus Z|\leq 2$, then $X$ is totally disconnected. Evidently,
this remains valid for every $Z\subset X$ with
$|X\setminus Z|<\omega$. Regarding hyperspace selections avoiding
closed sets or $F_\sigma$-sets, the interested reader is referred to
\cite{MR3712970}, see \cite[Theorem 5.4 and Question 5]{MR3712970}.


\end{document}